\documentclass[a4paper,11pt]{amsart}    
\usepackage{latexsym, amsmath, amsthm, amssymb, setspace,verbatim}
\usepackage{hyperref, aliascnt}
\usepackage{enumitem}

\title{Equivariant property (SI) revisited, II}
\author{G{\'a}bor Szab{\'o}}

\address{KU Leuven, Department of Mathematics, Celestijnenlaan 200b box 2400, \phantom{-..-}B-3001 Leuven, Belgium}
\email{gabor.szabo@kuleuven.be}
\thanks{This work was supported by internal KU Leuven BOF project C14/19/088 and project G085020N funded by the Research Foundation Flanders (FWO)}
\dedicatory{In memory of Eberhard Kirchberg.}

\begin{document}

\newcommand{\IA}[0]{\mathbb{A}} \newcommand{\IB}[0]{\mathbb{B}}
\newcommand{\IC}[0]{\mathbb{C}} \newcommand{\ID}[0]{\mathbb{D}}
\newcommand{\IE}[0]{\mathbb{E}} \newcommand{\IF}[0]{\mathbb{F}}
\newcommand{\IG}[0]{\mathbb{G}} \newcommand{\IH}[0]{\mathbb{H}}
\newcommand{\II}[0]{\mathbb{I}} \renewcommand{\IJ}[0]{\mathbb{J}}
\newcommand{\IK}[0]{\mathbb{K}} \newcommand{\IL}[0]{\mathbb{L}}
\newcommand{\IM}[0]{\mathbb{M}} \newcommand{\IN}[0]{\mathbb{N}}
\newcommand{\IO}[0]{\mathbb{O}} \newcommand{\IP}[0]{\mathbb{P}}
\newcommand{\IQ}[0]{\mathbb{Q}} \newcommand{\IR}[0]{\mathbb{R}}
\newcommand{\IS}[0]{\mathbb{S}} \newcommand{\IT}[0]{\mathbb{T}}
\newcommand{\IU}[0]{\mathbb{U}} \newcommand{\IV}[0]{\mathbb{V}}
\newcommand{\IW}[0]{\mathbb{W}} \newcommand{\IX}[0]{\mathbb{X}}
\newcommand{\IY}[0]{\mathbb{Y}} \newcommand{\IZ}[0]{\mathbb{Z}}

\newcommand{\CA}[0]{\mathcal{A}} \newcommand{\CB}[0]{\mathcal{B}}
\newcommand{\CC}[0]{\mathcal{C}} \newcommand{\CD}[0]{\mathcal{D}}
\newcommand{\CE}[0]{\mathcal{E}} \newcommand{\CF}[0]{\mathcal{F}}
\newcommand{\CG}[0]{\mathcal{G}} \newcommand{\CH}[0]{\mathcal{H}}
\newcommand{\CI}[0]{\mathcal{I}} \newcommand{\CJ}[0]{\mathcal{J}}
\newcommand{\CK}[0]{\mathcal{K}} \newcommand{\CL}[0]{\mathcal{L}}
\newcommand{\CM}[0]{\mathcal{M}} \newcommand{\CN}[0]{\mathcal{N}}
\newcommand{\CO}[0]{\mathcal{O}} \newcommand{\CP}[0]{\mathcal{P}}
\newcommand{\CQ}[0]{\mathcal{Q}} \newcommand{\CR}[0]{\mathcal{R}}
\newcommand{\CS}[0]{\mathcal{S}} \newcommand{\CT}[0]{\mathcal{T}}
\newcommand{\CU}[0]{\mathcal{U}} \newcommand{\CV}[0]{\mathcal{V}}
\newcommand{\CW}[0]{\mathcal{W}} \newcommand{\CX}[0]{\mathcal{X}}
\newcommand{\CY}[0]{\mathcal{Y}} \newcommand{\CZ}[0]{\mathcal{Z}}

\newcommand{\FA}[0]{\mathfrak{A}} \newcommand{\FB}[0]{\mathfrak{B}}
\newcommand{\FC}[0]{\mathfrak{C}} \newcommand{\FD}[0]{\mathfrak{D}}
\newcommand{\FE}[0]{\mathfrak{E}} \newcommand{\FF}[0]{\mathfrak{F}}
\newcommand{\FG}[0]{\mathfrak{G}} \newcommand{\FH}[0]{\mathfrak{H}}
\newcommand{\FI}[0]{\mathfrak{I}} \newcommand{\FJ}[0]{\mathfrak{J}}
\newcommand{\FK}[0]{\mathfrak{K}} \newcommand{\FL}[0]{\mathfrak{L}}
\newcommand{\FM}[0]{\mathfrak{M}} \newcommand{\FN}[0]{\mathfrak{N}}
\newcommand{\FO}[0]{\mathfrak{O}} \newcommand{\FP}[0]{\mathfrak{P}}
\newcommand{\FQ}[0]{\mathfrak{Q}} \newcommand{\FR}[0]{\mathfrak{R}}
\newcommand{\FS}[0]{\mathfrak{S}} \newcommand{\FT}[0]{\mathfrak{T}}
\newcommand{\FU}[0]{\mathfrak{U}} \newcommand{\FV}[0]{\mathfrak{V}}
\newcommand{\FW}[0]{\mathfrak{W}} \newcommand{\FX}[0]{\mathfrak{X}}
\newcommand{\FY}[0]{\mathfrak{Y}} \newcommand{\FZ}[0]{\mathfrak{Z}}

\newcommand{\eins}[0]{\textbf{1}}
\newcommand\set[1]{\left\{#1\right\}}  
\renewcommand{\phi}[0]{\varphi}
\newcommand{\eps}[0]{\varepsilon}
\newcommand{\cstar}[0]{\ensuremath{\mathrm{C}^*}}
\newcommand{\cc}[0]{\ensuremath{\simeq_{\mathrm{cc}}}}
\newcommand{\id}[0]{\ensuremath{\operatorname{id}}}
\newcommand{\dist}[0]{\ensuremath{\operatorname{dist}}}
\newcommand{\dst}[0]{\displaystyle}
\newcommand{\supp}[0]{\operatorname{supp}}

\newtheorem{satz}{Satz}[section]		

\newaliascnt{corCT}{satz}
\newtheorem{corollary}[corCT]{Corollary}
\aliascntresetthe{corCT}
\providecommand*{\corCTautorefname}{Corollary}
\newaliascnt{lemmaCT}{satz}
\newtheorem{lemma}[lemmaCT]{Lemma}
\aliascntresetthe{lemmaCT}
\providecommand*{\lemmaCTautorefname}{Lemma}
\newaliascnt{propCT}{satz}
\newtheorem{proposition}[propCT]{Proposition}
\aliascntresetthe{propCT}
\providecommand*{\propCTautorefname}{Proposition}
\newaliascnt{theoremCT}{satz}
\newtheorem{theorem}[theoremCT]{Theorem}
\aliascntresetthe{theoremCT}
\providecommand*{\theoremCTautorefname}{Theorem}
\newtheorem*{theoreme}{Theorem}
\newtheorem*{core}{Corollary}

\theoremstyle{definition}

\newaliascnt{conjectureCT}{satz}
\newtheorem{conjecture}[conjectureCT]{Conjecture}
\aliascntresetthe{conjectureCT}
\providecommand*{\conjectureCTautorefname}{Conjecture}
\newaliascnt{defiCT}{satz}
\newtheorem{definition}[defiCT]{Definition}
\aliascntresetthe{defiCT}
\providecommand*{\defiCTautorefname}{Definition}
\newtheorem*{defie}{Definition}
\newaliascnt{notaCT}{satz}
\newtheorem{notation}[notaCT]{Notation}
\aliascntresetthe{notaCT}
\providecommand*{\notaCTautorefname}{Notation}
\newtheorem*{notae}{Notation}
\newaliascnt{remCT}{satz}
\newtheorem{remark}[remCT]{Remark}
\aliascntresetthe{remCT}
\providecommand*{\remCTautorefname}{Remark}
\newtheorem*{reme}{Remark}
\newaliascnt{exampleCT}{satz}
\newtheorem{example}[exampleCT]{Example}
\aliascntresetthe{exampleCT}
\providecommand*{\exampleCTautorefname}{Example}
\newaliascnt{questionCT}{satz}
\newtheorem{question}[questionCT]{Question}
\aliascntresetthe{questionCT}
\providecommand*{\questionCTautorefname}{Question}

\newcounter{theoremintro}
\renewcommand*{\thetheoremintro}{\Alph{theoremintro}}
\newaliascnt{theoremiCT}{theoremintro}
\newtheorem{theoremi}[theoremiCT]{Theorem}
\aliascntresetthe{theoremiCT}
\providecommand*{\theoremiCTautorefname}{Theorem}
\newaliascnt{coriCT}{theoremintro}
\newtheorem{cori}[coriCT]{Corollary}
\aliascntresetthe{coriCT}
\providecommand*{\coriCTautorefname}{Corollary}
\newaliascnt{conjectureiCT}{theoremintro}
\newtheorem{conjecturei}[conjectureiCT]{Conjecture}
\aliascntresetthe{conjectureiCT}
\providecommand*{\conjectureiCTautorefname}{Conjecture}

\numberwithin{equation}{section}
\renewcommand{\theequation}{e\thesection.\arabic{equation}}


\begin{abstract}
We investigate Matui--Sato's notion of property (SI) for \cstar-dynamics, this time with a focus on actions of possibly non-amenable groups.
The main result is a generalization of earlier work:
For any countable group $\Gamma$ and any non-elementary separable simple nuclear \cstar-algebra $A$ with strict comparison, every amenable $\Gamma$-action on $A$ has equivariant property (SI).
This is deduced from a more general statement involving relative property (SI) for certain inclusions into ultraproducts.
The article concludes with a few consequences of this result.
\end{abstract}

\maketitle


\section*{Introduction}

This article aims to advance the structure theory for actions of countable groups on well-behaved \cstar-algebras.
In particular the present work is inspired by new developments triggered by the notion of amenability for actions of nonamenable groups on \cstar-algebras; see \cite{Suzuki19, Suzuki21, BussEchterhoffWillett20, BussEchterhoffWillett23, BeardenCrann22, OzawaSuzuki21, GabeSzabo23kp, Suzuki23, Suzuki23-2}.
By now it stands to reason that many of the previously known structural properties that have been shown for actions of amenable groups can be expected to have some analog for amenable actions of not necessarily amenable groups, which may lead to some kind of classification theory in the long run.

In this article we revisit Matui--Sato's notion of equivariant property (SI) from this point of view.
Historically, property (SI) (abbreviating ``small isometries'') is an analog of Kirchberg--Phillips' famous observation \cite{KirchbergPhillips00} that the central sequence algebra of a Kirchberg algebra is simple and purely infinite (see also \cite{Kirchberg04}), except that one takes into account possibly existing traces on the underlying \cstar-algebra.
It rose to prominence in the structure theory of \cstar-algebras about a decade ago through various impactful applications \cite{MatuiSato12acta, MatuiSato14UHF, KirchbergRordam14, TomsWhiteWinter15, SatoWhiteWinter15, Winter16, BBSTWW, CETWW, CastillejosEvington20} in the Toms--Winter conjecture for separable unital simple nuclear \cstar-algebras; see also \cite{ElliottToms08, TomsWinter09, WinterZacharias10}.
Around the same time Matui and Sato additionally considered a generalization of their concept to actions of discrete groups \cite{Sato10, MatuiSato12, MatuiSato14}, which led to the first results asserting that certain actions of amenable groups on classifiable \cstar-algebras automatically absorb the trivial action on the Jiang--Su algebra $\CZ$ tensorially, as well as some initial classification of group actions on the Jiang--Su algebra itself.
Other consequences of property (SI) related to Rokhlin-type properties were explored in \cite{Liao16, Liao17, Szabo19ssa4, GardellaHirshbergVaccaro22}.

In the context of structure and classification, the Jiang--Su algebra \cite{JiangSu99} has become a cornerstone object as an infinite-dimensional \cstar-algebra that abstractly resembles the algebra of complex numbers $\IC$.
We refer to the prequel article \cite{Szabo21si} for a more thorough explanation of the relevance of Jiang--Su stability and related references.
Equivariant property (SI) is still the quintessential gadget for verifying that certain \cstar-dynamics are equivariantly Jiang--Su stable; see more recently \cite{Sato19, GardellaHirshbergVaccaro22, Wouters23, SzaboWouters23-2}.
We also note that in more restricted cases, it has been used as an ingredient to prove even more powerful classification theorems such as \cite{Nawata21, Nawata23, Nawata23preprint}.

As a follow-up to various results asserting that certain group actions automatically have equivariant property (SI), I showed in \cite{Szabo21si} that every action of an amenable group on any (possibly non-unital) classifiable \cstar-algebra has equivariant property (SI).
The main result of this note is the observation that the same holds true for amenable actions of arbitrary countable groups; see \autoref{thm:equi-SI} and \autoref{cor:equivariant-SI}.
Note that the existence of amenable actions on finite classifiable \cstar-algebras has recently been established by Suzuki in \cite{Suzuki23, Suzuki23-2}, though at the moment the existence of such examples on unital \cstar-algebras remains open. 
While utilizing in large part the approach laid out in \cite{Szabo21si}, the new conceptual part of the proof of our main result is an adaptation of an averaging trick from \cite{GabeSzabo23kp} to replace the use of F{\o}lner sets in amenable groups.
The main argument to make this work turns out to be fairly straightforward (though nontrivial) if the group action is outer, but needs an extra twist if the group action is made up of both outer and inner automorphisms.
In the general case, our proof of equivariant property (SI) makes use of a slightly new characterization of amenability, proven in the second section, that takes into account the normal subgroup of elements acting via inner automorphisms.
The proof of this new characterization uses an adapted version of the Hahn--Banach trick from the main proof in Ozawa--Suzuki's recent work on amenability and the quasicentral approximation property \cite{OzawaSuzuki21}.

This note ends with two direct corollaries of the main result that generalize known results for actions of amenable groups.
We let $\Gamma$ be a countable group.
In \autoref{cor:Kirchberg}, it is shown that every amenable $\Gamma$-action on a Kirchberg algebra $A$ is equivariantly $\CO_\infty$-absorbing.
We note that this fact could be deduced from results in \cite{GabeSzabo23kp} for outer actions, but not for general actions in a direct way.
In \autoref{cor:equi-Z-stab}, it is shown that a given amenable $\Gamma$-action on a finite classifiable \cstar-algebra $A$ is equivariantly Jiang--Su stable if and only if for the induced $\Gamma$-action on the uniform tracial central sequence algebra $A^\omega\cap A'$, the fixed point algebra $(A^{\omega}\cap A')^\Gamma$ contains a unital copy of any finite matrix algebra.
It remains an open problem to what extent the latter condition is redundant (even for $\Gamma$ amenable \cite{SzaboWouters23-2}), though this needs to be tackled with different methods from the ones treated in this note.
In light of the fact that the analogous phenomenon holds in the realm of von Neumann algebras \cite{SzaboWouters23-1}, this gives rise to a question about a generalization of \cite[Conjecture A]{Szabo21si}: Is it true that every amenable $\Gamma$-action on any classifiable \cstar-algebra is equivariantly Jiang--Su stable?

In private correspondence, the authors of \cite{GGKNV23} informed me of some striking observations related to equivariant property (SI) and their notion of tracial amenability, which will likely appear in their subsequent work.


\section{Preliminaries}

Most of the notation and concepts given in this section are recalled directly from \cite{Szabo21si}.
In this article we assume familiarity with the cone of lower semi-continuous traces $T(A)$ on a \cstar-algebra $A$; see \cite{ElliottRobertSantiago11} for a detailed exposition.
We will at times make use of the Pedersen ideal $\CP(A)$ of a \cstar-algebra $A$ and its basic properties \cite[Section 5.6]{Pedersen} without explicit reference.
We say that a \cstar-algebra is non-elementary, if it is not isomorphic to the algebra of compact operators on some Hilbert space.
When $a$ and $b$ are some elements in a \cstar-algebra, we may write $a=_\eps b$ as short-hand for $\|a-b\|\leq\eps$.
We will throughout denote by $\omega$ a free ultrafilter on $\IN$.
We say that a mathematical statement involving a natural number $n$ holds ``for $\omega$-all $n$'' if the set of all $n\in\IN$ for which the statement is true belongs to $\omega$.

\begin{notation}
Let $A$ be a simple \cstar-algebra.
In the context of considering the tracial cone $T(A)$, we use the symbol $0$ for the zero trace, and the symbol $\infty$ for the trace taking the value $\infty$ on all non-zero positive elements; these are the trivial traces on $A$.
Recall that a non-trivial lower semi-continuous trace $\tau$ on $A$ is automatically faithful and densely defined.
We consider the corresponding dimension function $d_\tau: A_+\to [0,\infty]$ given by $d_\tau(a)=\lim_{n\to\infty} \tau(f_n(a))$, where $f_n: [0,\infty)\to [0,1]$ is any pointwise-increasing sequence of continuous functions with $f_n(0)=0$ and $\lim_{n\to\infty} f_n(t)=1$ for all $t>0$. 

A compact generator $K$ of $T(A)$ is a compact subset $K\subset T(A)\setminus\set{0,\infty}$ such that $\IR^{>0}\cdot K = T(A)\setminus\set{0,\infty}$.
We call $A$ traceless, if there are no non-trivial traces on $A$, in which case we explicitly declare the empty set to be a compact generator of $T(A)$.
\end{notation}

\begin{definition}
A simple \cstar-algebra $A$ is said to have local strict comparison\footnote{The better-known property called strict comparison calls for the stabilization $A\otimes\CK$ to have local strict comparison.}, if the following is true.
Whenever two non-zero positive elements $a,b\in \mathcal P(A)_+$ in the Pedersen ideal satisfy $d_\tau(a) < d_\tau(b)$ for all $\tau\in T(A)\setminus\set{0,\infty}$, then it follows that $a\precsim b$, i.e., there is a sequence $r_n\in A$ with $r_n^*br_n\to a$.
\end{definition}

\begin{remark} \label{rem:traceless-comparison}
A simple traceless \cstar-algebra has local strict comparison if and only if it is purely infinite.
Thus, a separable simple nuclear traceless \cstar-algebra with local strict comparison is precisely a Kirchberg algebra.
\end{remark}

\begin{notation}[cf.\ \cite{Kirchberg04}]
For a specified sequence of \cstar-algebras $B_n$ and a free ultrafilter $\omega$ on $\IN$, we will denote by
\[
\CB_\omega = \prod_{n\in\IN} B_n / \set{ (b_n)_n \mid \lim_{n\to\omega} \|b_n\|=0 }
\]
their ultraproduct \cstar-algebra.
If $\Gamma$ is a discrete group and $\beta_n:\Gamma\curvearrowright B_n$ is a sequence of actions, we furthermore denote by $\beta_\omega:\Gamma\curvearrowright\CB_\omega$ the induced ultraproduct action.

For an inclusion of a (usually separable) \cstar-algebra $A\subset\CB_\omega$, we write
\[
\CB_\omega\cap A' = \set{ x\in\CB_\omega \mid [x,A]=0 },\quad \CB_\omega\cap A^\perp = \set{x\in\CB_\omega \mid xA=Ax=0},
\]
and define
\[
F(A,\CB_\omega) = (\CB_\omega\cap A')/(\CB_\omega\cap A^\perp).
\]
If furthermore $A$ is $\beta_\omega$-invariant, then so are $\CB_\omega\cap A'$ and $\CB_\omega\cap A^\perp$, and we get an induced action on the quotient $\tilde{\beta}_\omega:\Gamma\curvearrowright F(A,\CB_\omega)$.
In the special case where $B_n=A$ for all $n$ and the inclusion $A\subset A_\omega$ is the obvious one, we abbreviate $F_\omega(A)=F(A,A_\omega)$.
\end{notation}

\begin{definition} \label{def:generalized-limit-trace}
Let $B_n$ be a sequence of \cstar-algebras and $\omega$ a free ultrafilter on $\mathbb N$.
For a sequence $\tau_n$ of lower semi-continuous traces on $B_n$, we may define a lower semi-continuous trace $\tau^\omega: \Big( \prod_{n\in\mathbb N} B_n \Big)_+\to [0,\infty]$ via
\[
\tau^\omega( (b_n)_n ) = \sup_{\eps>0} \lim_{n\to\omega} \tau_n\big( (b_n-\eps)_+ \big),
\]
where $b_n\in A$ is a bounded sequence of positive elements.\footnote{See \cite[Lemma 3.1]{ElliottRobertSantiago11}, which implies that this map indeed yields a lower semi-continuous trace.}
We see that $\tau^\omega( (b_n)_n ) = 0$ whenever $\lim_{n\to\omega} \|b_n\| = 0$, so $\tau^\omega$ induces a lower semi-continuous trace on $\CB_\omega$, which we also denote by $\tau^\omega$ with slight abuse of notation.
A trace of this form on $\CB_\omega$ is called a generalized limit trace.
\end{definition}

\begin{remark} \label{rem:traces-on-FA}
With the same assumptions as in \autoref{def:generalized-limit-trace}, let additionally $A\subseteq \CB_\omega$ be a \cstar-subalgebra.
Suppose that $\tau^\omega$ is a generalized limit trace on $\CB_\omega$.
Let $a\in A_+$ be a positive element with $\tau^\omega(a)<\infty$.
The assignment $\tau^\omega_a: (\CB_\omega\cap A')_+\to [0,\infty]$, $x\mapsto\tau^\omega(xa)$ defines a bounded trace of norm equal to $\tau^\omega(a)$.
It is clear that $\tau^\omega_a(x)=0$ for all $x\in \CB_\omega\cap A^\perp$.
This induces a trace on $F(A, \CB_\omega)$, which we again denote $\tau^\omega_a$, and which has the same norm.
\end{remark}

\begin{definition} \label{def:null-full}
Let $B_n$ be a sequence of simple \cstar-algebras with ultraproduct $\CB_\omega$, and let $A\subset \CB_\omega$ be a non-zero simple \cstar-subalgebra.
\begin{enumerate}[label=\textup{(\roman*)},leftmargin=*]
\item \label{def:null-full:1} 
We say that a positive contraction $f\in \CB_\omega\cap A'$ or $f\in F(A, \CB_\omega)$ is tracially supported at 1, if either one of the following is true:
Either $B_n$ is traceless for $\omega$-all $n$, in which case we demand $\|fa\|=\|a\|$ for all $a\in A_+$.
Or, $B_n$ has non-trivial traces for $\omega$-all $n$, in which case we demand the following:
For every (or any) non-zero positive element $a\in\CP(A)$, there exists a constant $\kappa=\kappa(f,a)>0$ such that for every generalized limit trace $\tau^\omega$ with $0<\tau^\omega(a)<\infty$, one has $\displaystyle \inf_{k\in\mathbb N} \tau^\omega_a(f^k) \geq \kappa\tau^\omega(a)$.
\item \label{def:null-full:2} 
We say that a positive element $e\in \CB_\omega\cap A'$ or $e\in F(A, \CB_\omega)$ is tracially null, if $\tau^\omega_a(e)=0$ for every positive element $a\in\CP(A)$ and every generalized limit trace $\tau^\omega$ on $\CB_\omega$ with $\tau^\omega(a)<\infty$.
\end{enumerate}
\end{definition}

The next two definitions have their origin in \cite{Sato10, MatuiSato12, MatuiSato12acta, MatuiSato14}:

\begin{definition} \label{def:property-SI}
Let $B_n$ be a sequence of simple \cstar-algebras with ultraproduct $\CB_\omega$.
Suppose that $\Gamma$ is a countable discrete group and $\beta_n: \Gamma\curvearrowright B_n$ is a sequence of actions giving rise to the ultraproduct action $\beta_\omega: \Gamma\curvearrowright\CB_\omega$.
Let $A\subset\CB_\omega$ be a non-zero separable simple $\beta_\omega$-invariant \cstar-subalgebra.
We say that the inclusion $A \subset \CB_\omega$ has equivariant property (SI) relative to $\beta_\omega$ if the following holds:

Whenever $e,f\in F(A,\CB_\omega)^{\tilde{\beta}_\omega}$ are two positive contractions  such that $f$ is tracially supported at 1 and $e$ is tracially null, there exists a contraction $s\in F(A,\CB_\omega)^{\tilde{\beta}_\omega}$ with $fs=s$ and $s^*s=e$.

More specifically, we say that a separable simple \cstar-algebra $A$ has property (SI) relative to an action $\alpha: \Gamma\curvearrowright A$, if the canonical inclusion $A\subset A_\omega$ has property (SI) relative to the ultrapower action $\alpha_\omega$.
\end{definition}


\section{The quasicentral approximation property relative to a subgroup}

\begin{notation}
Let $A$ be a \cstar-algebra and $I$ a non-empty set.
We consider the space of finitely supported maps $I\to A$, denoted $c_{00}(I,A)$, equipped with the $A$-valued inner product given by $\langle \xi_1 \mid \xi_2 \rangle = \sum_{j\in I} \xi_1(j)^*\xi_2(j)$.
The resulting norm on $c_{00}(I,A)$ shall be denoted as $\|\cdot\|_2$ and $\ell^2(I,A)$ is defined as the $\|\cdot\|_2$-closure of $c_{00}(I,A)$.
Given a state $\phi$ on $A$, we denote by $\|\cdot\|_\phi$ the seminorm on $\ell^2(I,A)$ given via $\|\eta \|_\phi=\phi(\langle\eta\mid\eta\rangle)^{1/2}$.
For notational convenience, we sometimes write $|\eta|=\langle\eta\mid\eta\rangle^{1/2}\in A_+$ for $\eta\in\ell^2(I,A)$.
\end{notation}

\begin{lemma} \label{lem:ell2-Kaplansky}
Let $A$ be a \cstar-algebra and $I$ a non-empty set.
Given any contraction $\xi\in\ell^2(I,A^{**})$, there exists a net of contractions $\zeta_i\in c_{00}(I,A)$ with $\|\xi-\zeta_i\|_{\phi}\to 0$ for every state $\phi$ on $A$.
\end{lemma}
\begin{proof}
This is explained in the last paragraph of \cite[Subsection 2.6]{OzawaSuzuki21} or it can be  deduced as a consequence from \cite[Corollary 4.6]{BussEchterhoffWillett23}.
\end{proof}

\begin{definition}[see {\cite[Definition 4.1]{Delaroche87}}]
Let $\Gamma$ be a countable group.
Let $\alpha: \Gamma\curvearrowright A$ be an action on a \cstar-algebra.
We consider the induced action $\alpha^{**}: \Gamma\curvearrowright A^{**}$ and the action $\bar{\alpha}^{**}: \Gamma\curvearrowright\ell^\infty(\Gamma,A^{**})$ given by $\bar{\alpha}^{**}_g(f)(h)=\alpha^{**}_g(f(g^{-1}h))$.
We say that $\alpha$ is amenable, if $\alpha^{**}$ is von Neumann amenable, i.e., there exists a $\Gamma$-equivariant conditional expectation
\[
\ell^\infty(\Gamma,\CZ(A^{**})) \to \CZ(A^{**}).
\]
\end{definition}

\begin{definition}[cf.\ \cite{KadisonRingrose67}]
Let $\alpha$ be an automorphism on a \cstar-algebra $A$.
We say that $\alpha$ is universally weakly inner, if the induced automorphism $\alpha^{**}$ on the double dual $A^{**}$ is inner.
\end{definition}

\begin{notation}
Let $\alpha: \Gamma\curvearrowright A$ be an action on a \cstar-algebra.
Let $N\subseteq\Gamma$ be a subgroup.
We shall write $\bar{g}=gN\in\Gamma/N$ for $g\in\Gamma$ to lighten notation.
We denote by $\bar{\alpha}: \Gamma\curvearrowright \ell^2(\Gamma/N,A)$ the isometric action given by $\bar{\alpha}_g(\xi)(\bar{h})=\alpha_g\big(\xi(\overline{g^{-1}h})\big)$ for all $g,h\in\Gamma$ and $\xi\in\ell^2(\Gamma/N,A)$.
\end{notation}

The proof of the following lemma borrows the Hahn--Banach argument from the proof of \cite[Theorem 3.2]{OzawaSuzuki21}.
The lemma itself is a slight generalization of part of that theorem for discrete groups, or can be seen as a generalization of the same result for discrete groups that has appeared independently in a prior preprint version of \cite{BussEchterhoffWillett23}.

\begin{lemma} \label{lem:relative-QAP}
Let $\alpha: \Gamma\curvearrowright A$ be an action on a \cstar-algebra.
Let $N\subseteq\Gamma$ be a normal subgroup so that $\alpha_g$ is universally weakly inner for all $g\in N$.
If $\alpha$ is amenable, then there exists a net of contractions $\zeta_i\in\ell^2(\Gamma/N,A)$ satisfying
\[
\langle\zeta_i\mid\zeta_i\rangle a\to a,\quad \|a \zeta_i-\zeta_i a\|_2\to 0,\quad \|\bar{\alpha}_g(\zeta_i)-\zeta_i\|_2\to 0
\]
for all $a\in A$ and $g\in\Gamma$.
\end{lemma}
\begin{proof}
First we note that the claim is easy if $N$ is a subgroup of finite index.
In that case, pick an increasing and approximately $\alpha$-invariant approximate unit $e_i\in A$ and check with a straightforward calculation that the net $\zeta_i(\bar{g})=|\Gamma/N|^{-1/2}e_i$ does the job.
Henceforth, we will thus assume that $N$ has infinite index in $\Gamma$.

We can consider the unital embedding 
\[
\iota_N: \ell^\infty(\Gamma/N,\CZ(A^{**}))\to\ell^\infty(\Gamma,\CZ(A^{**}))
\] 
given by $\iota_N(f)(g)=f(\bar{g})$ for all $f\in\ell^\infty(\Gamma/N,\CZ(A^{**}))$ and $g\in\Gamma$.
The image of $\iota_N$ is clearly globally $\bar{\alpha}^{**}$-invariant and the induced action on $\ell^\infty(\Gamma/N,\CZ(A^{**}))$ is denoted also as $\bar{\alpha}^{**}$ with slight abuse of notation.
Since $\alpha$ is amenable, there exists a $\Gamma$-equivariant conditional expectation 
\[
E: \ell^\infty(\Gamma,\CZ(A^{**}))\to\CZ(A^{**}).
\]
The composition $E\circ\iota_N$ yields a $\Gamma$-equivariant conditional expectation 
\[
\ell^\infty(\Gamma/N,\CZ(A^{**}))\to\CZ(A^{**}).
\]
By definition, the restricted action $\alpha^{**}|_N$ is trivial on $\CZ(A^{**})$, which implies that $N$ acts trivially on both sides above.
We can thus identify $\alpha^{**}$ on $\CZ(A^{**})$ with a $\Gamma/N$-action after modding out $N$, which we have just argued yields an amenable action of $\Gamma/N$.

By \cite[Theorem 3.3]{Delaroche87}, there exists a net of contractions $\xi_i\in c_{00}(\Gamma/N,\CZ(A^{**}))$ such that
\[
\langle \xi_i \mid \bar{\alpha}^{**}_g(\xi_i)\rangle \to \eins \text{ ultraweakly for all } g\in\Gamma.
\]
From this we can conclude $\langle\xi_i\mid\xi_i\rangle\to \eins$ ultraweakly and
\[
\begin{array}{cl}
\multicolumn{2}{l}{ \langle \bar{\alpha}^{**}_g(\xi_i)-\xi_i \mid \bar{\alpha}^{**}_g(\xi_i)-\xi_i\rangle }\\
=& \langle\xi_i\mid\xi_i\rangle + \alpha^{**}_g(\langle\xi_i\mid\xi_i\rangle) - \langle\xi_i\mid\bar{\alpha}^{**}_g(\xi_i)\rangle - \langle\xi_i\mid\bar{\alpha}^{**}_g(\xi_i)\rangle^* \\
\to& \eins+\eins-\eins-\eins = 0
\end{array} 
\]
ultraweakly for any $g\in\Gamma$.
Let $\rho_1,\dots,\rho_n$ be finitely many states on $A$, which we view as normal states on $A^{**}$.
Let $F\subset A$ and $K\subset\Gamma$ be finite subsets.
Let $\eps>0$.
For sufficiently large $i$, we can conclude
\[
\rho_j(\langle\xi_i\mid\xi_i\rangle)\geq 1-\frac{\eps}{2} \quad\text{and}\quad \|\bar{\alpha}^{**}_g(\xi_i)-\xi_i \|_{\rho_j}\leq\eps/2
\]
for all $g\in K$ and $j=1,\dots,n$.
Using \autoref{lem:ell2-Kaplansky}, we can find a contraction $\eta\in c_{00}(\Gamma/N,A)$ satisfying
\[
\rho_j(\langle \eta \mid \eta\rangle)\geq 1-\eps, \quad \|\bar{\alpha}_g(\eta)-\eta \|_{\rho_j}\leq\eps,\quad \|a\eta-\eta a\|_{\rho_j}\leq\eps
\]
for all $g\in K$, $a\in F$ and $j=1,\dots,n$.
Since the tuple of states $(\rho_1,\dots,\rho_n)$ and $\eps>0$ was arbitrary, we can find a net of contractions $\eta_j\in c_{00}(\Gamma/N,A)$ such that
\[
q_j:=\sum_{g\in K} |\bar{\alpha}_g(\eta_j)-\eta_j|^2 + \sum_{a\in F} \Big( a^*(\eins-\langle\eta_j\mid\eta_j\rangle)a + |a \eta_j-\eta_j a|^2 \Big) \longrightarrow 0
\]
in the weak topology.
Now fix $\eps>0$ again.
By the Hahn--Banach theorem, we may find $\lambda_1,\dots,\lambda_n\in [0,1]$ with $\sum_{k=1}^n\lambda_k=1$ and indices $j_1,\dots,j_n$ such that $\sum_{k=1}^n \lambda_k q_{j_k} \leq \eps$.
Using that $\Gamma/N$ is infinite, we may pick elements $t_1,\dots,t_n\in\Gamma/N$ such that
\begin{equation} \label{eq:disjoint-supports}
\big[ \overline{K}\cdot\supp(\eta_{j_{k_1}}) t_{k_1}^{-1} \big] \cap\big[ \overline{K}\cdot\supp(\eta_{j_{k_2}}) t_{k_2}^{-1}\big] = \emptyset
\end{equation}
for all $1\leq k_1\neq k_2 \leq n$.
We define $\zeta\in c_{00}(\Gamma/N,A)$ via
\[
\zeta(\bar{g})=\sum_{k=1}^n \lambda_k^{1/2}\eta_{j_k}(\bar{g}t_k),\quad g\in\Gamma.
\]
We observe with conditition \eqref{eq:disjoint-supports} that
\[
\begin{array}{ccl}
\langle\zeta\mid\zeta\rangle &=& \displaystyle\sum_{\bar{g}\in\Gamma/N} \Big| \sum_{k=1}^n \lambda_k^{1/2}\eta_{j_k}(\bar{g}t_k) \Big|^2 \\
&=& \displaystyle\sum_{k=1}^n \lambda_k \sum_{\bar{g}\in\Gamma/N} |\eta_{j_k}(\bar{g})|^2 \\
&=& \displaystyle \sum_{k=1}^n \lambda_k \langle \eta_{j_k}\mid\eta_{j_k}\rangle.
\end{array}
\]
Thus it follows for $a\in F$ that
\[
a^*(\eins-\langle\zeta\mid\zeta\rangle)a = \sum_{k=1}^n \lambda_k a^*(\eins-\langle \eta_{j_k}\mid\eta_{j_k}\rangle) a \leq \sum_{k=1}^n \lambda_k q_{j_k} \leq \eps.
\]
Using the \cstar-identity one sees
\[
\|(\eins-\langle\zeta\mid\zeta\rangle) a\|^2 \leq \|(\eins-\langle\zeta\mid\zeta\rangle)^{1/2} a\|^2 = \|a^*(\eins-\langle\zeta\mid\zeta\rangle)a\|\leq\eps.
\]
In a similar fashion one sees
\[
|a\zeta-\zeta a|^2 = \sum_{k=1}^n \lambda_k |a\eta_{j_k}-\eta_{j_k}a|^2 \leq \sum_{k=1}^n \lambda_k q_{j_k} \leq\eps
\]
for all $a\in F$, and
\[
|\bar{\alpha}_g(\zeta)-\zeta|^2 = \sum_{k=1}^n \lambda_k |\bar{\alpha}_g(\eta_{j_k})-\eta_{j_k}|^2 \leq \sum_{k=1}^n \lambda_k q_{j_k} \leq \eps
\]
for all $g\in K$.
Since the triple $(F,K,\eps)$ was arbitrary, we may obtain a net as required by the claim.
\end{proof}


\section{Equivariant property (SI)}

The statement of the next lemma arises upon following the proof of \cite[Theorem 4.2]{Szabo21si} word for word until the third line of page 1218.
Although one assumes in that theorem that the group $\Gamma$ is amenable, we note that this assumption is only used in the later part of the proof we are not referring to here.

\begin{lemma} \label{lem:precursor-equi-SI}
Let $\Gamma$ be a countable group.
Let $B_n$ be a sequence of simple \cstar-algebras with local strict comparison, and let $\beta_n: \Gamma\curvearrowright B_n$ be a sequence of $\Gamma$-actions.
Let $\CB_\omega$ be the associated ultraproduct, and $\beta_\omega: \Gamma\curvearrowright\CB_\omega$ the corresponding ultraproduct action. 
Let $A$ be a non-elementary separable simple nuclear \cstar-algebra, and suppose that $A\subset\CB_\omega$ is an inclusion as a $\beta_\omega$-invariant \cstar-subalgebra.
Let $N\subseteq\Gamma$ be the normal subgroup of elements $g\in\Gamma$ such that the automorphism $\beta_{\omega,g}|_A$ on $A$ is inner.
Suppose that $\tilde{\beta}_{\omega,g}$ is trivial on $F(A,\CB_\omega)$ for all $g\in N$.

Suppose that $e,f\in F(A,\CB_\omega)^{\beta_\omega}$ are two positive contractions such that $e$ is tracially null and $f$ is tracially supported at 1.
Then there exists a contraction $s\in F(A,\CB_\omega)$ such that 
\[
fs=s,\ s^*s=e, \text{ and } s^*\beta_{\omega,g}(s)=0 \text{ for all } g\in\Gamma\setminus N.
\]
\end{lemma}

We are now ready to prove our main technical result, from which the result mentioned in the introduction follows afterwards.

\begin{theorem} \label{thm:equi-SI}
Let $\Gamma$ be a countable amenable group.
Let $B_n$ be a sequence of simple \cstar-algebras with local strict comparison, and let $\beta_n: \Gamma\curvearrowright B_n$ be a sequence of $\Gamma$-actions.
Let $\CB_\omega$ be the associated ultraproduct, and $\beta_\omega: \Gamma\curvearrowright\CB_\omega$ the corresponding ultraproduct action. 
Let $\alpha: \Gamma\curvearrowright A$ be an amenable action on a non-elementary separable simple nuclear \cstar-algebra, and suppose that $A\subset\CB_\omega$ is an equivariant inclusion with the following property: For every $g\in\Gamma$, if the automorphism $\alpha_g$ on $A$ is inner, then $\tilde{\beta}_{\omega,g}$ is trivial on $F(A,\CB_\omega)$.
Then the inclusion $A\subset\CB_\omega$ has property (SI) relative to $\beta_\omega$.
\end{theorem}
\begin{proof}
Let $N$ denote the normal subgroup of all elements $g\in\Gamma$ such that $\alpha_g$ is inner.
In order to verify the claim, let $e,f\in F(A,\CB_\omega)^{\tilde{\beta}_\omega}$ be two positive contractions such that $e$ is tracially null and $f$ is tracially supported at 1.
By \autoref{lem:precursor-equi-SI}, we can find a contraction $s\in F(A,\CB_\omega)$ having the property that
\begin{equation} \label{eq:input-s}
fs=s,\ s^*s=e, \text{ and } s^*\tilde{\beta}_{\omega,g}(s)=0 \text{ for all } g\in\Gamma\setminus N.
\end{equation}
Let $F=F^*\subset A$ and $K\subset\Gamma$ be finite subsets and $\eps>0$.
Since $\alpha$ is amenable, it follows from \autoref{lem:relative-QAP} that there exists a finitely supported function $\zeta\in c_{00}(\Gamma/N,A)$ satisfying
\begin{equation} \label{eq:amen-net}
\|\zeta\|_2\leq 1,\quad \|(\eins - \langle\zeta \mid \zeta \rangle) a\|\leq\eps,\quad \|a \zeta-\zeta a\|_2\leq\eps,\quad \|\bar{\alpha}_g(\zeta)-\zeta\|_2\leq\eps
\end{equation}
for all $a\in F$ and $g\in K$.
We consider the element $S\in\ell^\infty(\Gamma/N,\CB_\omega)$ given by $S(\bar{g})=\tilde{\beta}_{\omega,g}(s)$.
By the assumption on $N$, this is well defined.
We define a contraction $s_0\in\CB_\omega$ via
\[
s_0 = \sum_{\bar{g}\in G/N} \tilde{\beta}_{\omega,g}(s)\zeta_{\bar{g}} .
\]
Since $e$ is fixed by $\tilde{\beta}_\omega$, the equation $s^*s=e$ implies $e=\tilde{\beta}_{\omega,g}(s^*s)$ for any $g\in\Gamma$.
We use this to compute for $a\in F$ that
\[
\begin{array}{ccl}
s_0^* a s_0 &=& \dst \sum_{\bar{g}_1,\bar{g_2}\in\Gamma/N} \zeta_{\bar{g}_1}^* \tilde{\beta}_{\omega,g_1}(s)^*a\tilde{\beta}_{\omega,g_2}(s) \zeta_{\bar{g}_2}  \\
&\stackrel{\eqref{eq:input-s}}{=}& \dst \sum_{\bar{g}\in\Gamma/N} \zeta_{\bar{g}}^* \tilde{\beta}_{\omega,g}(s^*s) a \zeta_{\bar{g}} \\
&=& \dst \sum_{\bar{g}\in\Gamma/N} \zeta_{\bar{g}}^* e a \zeta_{\bar{g}} \\
&=& \dst \langle\zeta\mid a\zeta\rangle e \ \stackrel{\eqref{eq:amen-net}}{=}_{\hspace{-2mm}2\eps} \ ae.
\end{array}
\]
Following part of the previous computation we also see that $s_0^*s_0\leq e$ upon inserting $\eins$ in place of $a$.
In particular, $s_0$ is a contraction.
Given that $f$ is fixed by $\tilde{\beta}_\omega$, the relation $fs=s$ implies $f\tilde{\beta}_{\omega,g}(s)=\tilde{\beta}_{\omega,g}(s)$ for all $g\in \Gamma$.
Since $\zeta$ takes values in $A$, this readily implies $fs_0=s_0$.
Moreover we compute for all $h\in K$ that
\[
\begin{array}{cl}
\multicolumn{2}{l}{ |s_0-\beta_{\omega,h}(s_0)|^2 }\\ 
=& \dst \Big| \sum_{\bar{g}\in\Gamma/N}  \tilde{\beta}_{\omega,g}(s) \zeta_{\bar{g}} -  \tilde{\beta}_{\omega,hg}(s) \alpha_h(\zeta_{\bar{g}}) \Big|^2 \\
=& \dst \Big| \sum_{\bar{g}\in\Gamma/N}  \tilde{\beta}_{\omega,g}(s) \big( \zeta_{\bar{g}} -  \alpha_h(\zeta_{\overline{h^{-1}g}}) \big) \Big|^2 \\
=& \dst \Big| \sum_{\bar{g}\in\Gamma/N}  \tilde{\beta}_{\omega,g}(s) \big( \zeta_{\bar{g}} -  \ \bar{\alpha}_h(\zeta)_{\bar{g}} \big) \Big|^2 \\
\stackrel{\eqref{eq:input-s}}{=}& \dst \sum_{\bar{g}\in\Gamma/N} \big( \zeta_{\bar{g}} -  \ \bar{\alpha}_h(\zeta)_{\bar{g}} \big)^* \tilde{\beta}_{\omega,g}(s^*s) \big( \zeta_{\bar{g}} -  \ \bar{\alpha}_h(\zeta)_{\bar{g}} \big) \\
=& \dst e\cdot \langle \zeta-\bar{\alpha}_h(\zeta) \mid \zeta-\bar{\alpha}_h(\zeta) \rangle \ \leq \ \eps^2.
\end{array}
\]
This readily implies $\|s_0-\beta_{\omega,h}(s_0)\|\leq\eps$.

Since the triple $(F,K,\eps)$ was arbitrary, we can apply Kirchberg's $\eps$-test (\cite[Lemma 3.1]{KirchbergRordam14}) and find a contraction $s_1\in (\CB_\omega)^{\beta_\omega}$ with $s_1^*s_1\leq e$ such that $fs_1=s_1$ and $s_1^*as_1=ae$ for all $a\in A$.
For every $a\in A$ it follows that
\[
\begin{array}{ccl}
(as_1-s_1a)^*(as_1-s_1a) &=& s_1^* a^*a s_1 - a^* s_1^* a s_1 - s_1^* a^* s_1 a + a^* s_1^* s_1 a \\
&=& a^*ae - a^*ae-a^*ea+a^*s_1^*s_1a \\
&=& a^*(s_1^*s_1-e)a \leq 0.
\end{array}
\]
Hence $s_1$ commutes with $A$ and we conclude $s_1^*s_1-e\in\CB_\omega\cap A^\perp$.
The resulting element $s=s_1+(\CB_\omega\cap A^\perp)$ then defines a contraction in $F(A,\CB_\omega)^{\tilde{\beta}_\omega}$ obeying the equality $fs=s$ and $s^*s=e$.
This finishes the proof.
\end{proof}

\begin{corollary} \label{cor:equivariant-SI}
Let $\Gamma$ be a countable group.
Let $A$ be a non-elementary separable simple nuclear \cstar-algebra with local strict comparison.
Then $A$ has equivariant property (SI) relative to every amenable action $\Gamma\curvearrowright A$.
\end{corollary}
\begin{proof}
This follows from \autoref{thm:equi-SI} applied to $B_n=A$, $\beta_n=\alpha$ and the canonical inclusion $A\subset A_\omega$.
The extra assumption about the inclusion holds automatically:
If $\alpha_g$ is inner, then $\tilde{\alpha}_\omega$ on $F_\omega(A)$ is trivial; see \cite[Remark 1.8]{Szabo18ssa}.
\end{proof}

In complete analogy to the case of actions of amenable groups, we can deduce consequences related to the equivariant absorption of strongly self-absorbing \cstar-algebras.

\begin{corollary}\label{cor:Kirchberg}
Let $A$ be a Kirchberg algebra and $\Gamma$ a countable group.
Then every amenable $\Gamma$-action on $A$ is equivariantly $\CO_\infty$-stable.
\end{corollary}
\begin{proof}
In light of \autoref{rem:traceless-comparison}, $A$ is nothing but a traceless separable simple nuclear \cstar-algebras with (local) strict comparison.
Let $\alpha:\Gamma\curvearrowright A$ be an arbitrary action.
We consider the matrix amplification $\alpha^{(2)}=\id\otimes\alpha: \Gamma\curvearrowright M_2\otimes A=M_2(A)$ and the equivariant diagonal inclusion
\[
(A,\alpha) \to (M_2(A)_\omega,\alpha^{(2)}_\omega),\quad a\mapsto \left(\begin{matrix} a & 0 \\ 0 & a \end{matrix}\right).
\]
Then this inclusion satisfies the assumption in \autoref{thm:equi-SI} and thus has equivariant property (SI) relative to $\alpha^{(2)}_\omega$.
In order to end up with a proof of the claim, one can from this point follow the proof of \cite[Theorem 5.1]{Szabo21si} word for word.
\end{proof}

\begin{definition}[see {\cite[Definition 4.5]{Szabo21si} and \cite[Definition 1.11]{SzaboWouters23-2}}]
Let $A$ be a separable simple nuclear \cstar-algebra with a non-trivial trace.
We define the trace-kernel ideal $\CJ_A$ inside $F_\omega(A)$ as those elements $x$ such that $x^*x$ is tracially null in the sense of \autoref{def:null-full}.
The quotient $A^\omega\cap A' = F_\omega(A)/\CJ_A$ is called the (uniform) tracial central sequence algebra.
If $\alpha: \Gamma\curvearrowright A$ is an action of a group, then $\CJ_A$ is clearly $\tilde{\alpha}_\omega$-invariant, which allows us to consider the induced action $\alpha^\omega: \Gamma\curvearrowright A^\omega\cap A'$.
\end{definition}

Our final application of the main result is a tracial characterization of equivariant Jiang--Su stability, which extends a known characterization for amenable acting groups:

\begin{corollary} \label{cor:equi-Z-stab}
Let $A$ be a non-elementary separable simple nuclear \cstar-algebra with local strict comparison.
Let $\alpha: \Gamma\curvearrowright A$ be an amenable action of a countable discrete group.
Then $\alpha$ is equivariantly $\CZ$-stable if and only if for every $n\geq 2$, there exists a unital $*$-homomorphism $M_n\to (A^\omega\cap A')^{\alpha^\omega}$.
\end{corollary}
\begin{proof}
This follows directly from \autoref{cor:equivariant-SI} and \cite[Theorem 5.4]{SzaboWouters23-2}.
\end{proof}


\bibliographystyle{gabor}
\bibliography{master}

\begin{thebibliography}{10}
\providecommand{\url}[1]{\texttt{#1}}
\providecommand{\urlprefix}{URL }

\bibitem{Delaroche87}
C.~Anantharaman-Delaroche: Syst{\`e}mes dynamiques non commutatifs et
  moyennabilit{\'e}.
\newblock Math. Ann. 279 (1987), pp. 297--315.

\bibitem{BeardenCrann22}
A.~Bearden, J.~Crann: Amenable dynamical systems over locally compact groups.
\newblock Ergodic Theory Dynam. Systems 42 (2022), no.~8, pp. 2468--2508.

\bibitem{BBSTWW}
J.~Bosa, N.~P. Brown, Y.~Sato, A.~Tikuisis, S.~White, W.~Winter: Covering
  dimension of \cstar-algebras and 2-coloured classification.
\newblock Mem. Amer. Math. Soc. 257 (2019).
\newblock 97 pp.

\bibitem{BussEchterhoffWillett20}
A.~Buss, S.~Echterhoff, R.~Willett: Injectivity, crossed products, and amenable
  group actions.
\newblock Contemp. Math. 749 (2020), pp. 105--138.

\bibitem{BussEchterhoffWillett23}
A.~Buss, S.~Echterhoff, R.~Willett: Amenability and weak containment for
  actions of locally compact groups on \cstar-algebras.
\newblock Mem. Amer. Math. Soc., to appear  (2023).
\newblock \urlprefix\url{https://arxiv.org/abs/2003.03469v7}.

\bibitem{CastillejosEvington20}
J.~Castillejos, S.~Evington: Nuclear dimension of simple stably projectionless
  \cstar-algebras.
\newblock Anal. PDE 13 (2020), pp. 2205--2240.

\bibitem{CETWW}
J.~Castillejos, S.~Evington, A.~Tikuisis, S.~White, W.~Winter: Nuclear
  dimension of simple \cstar-algebras.
\newblock Invent. Math. 224 (2021), pp. 245--290.

\bibitem{ElliottRobertSantiago11}
G.~A. Elliott, L.~Robert, L.~Santiago: The cone of lower semicontinuous traces
  on a \cstar-algebra.
\newblock Amer. J. Math. 133 (2011), no.~4, pp. 969--1005.

\bibitem{ElliottToms08}
G.~A. Elliott, A.~S. Toms: Regularity properties in the classification program
  for separable amenable \cstar-algebras.
\newblock Bull. Amer. Math. Soc. 45 (2008), pp. 229--245.

\bibitem{GabeSzabo23kp}
J.~Gabe, G.~Szab{\'o}: The dynamical {K}irchberg--{P}hillips theorem.
\newblock Acta Math., to appear  (2023).
\newblock \urlprefix\url{https://arxiv.org/abs/2205.04933}.

\bibitem{GGKNV23}
E.~Gardella, S.~Geffen, J.~Kranz, P.~Naryshkin, A.~Vaccaro: Tracially amenable
  actions and purely infinite crossed products
  \urlprefix\url{https://arxiv.org/abs/2211.16872}.

\bibitem{GardellaHirshbergVaccaro22}
E.~Gardella, I.~Hirshberg, A.~Vaccaro: Strongly outer actions of amenable
  groups on {$\CZ$}-stable nuclear \cstar-algebras.
\newblock J. Math. Pures Appl. 162 (2022), pp. 76--123.

\bibitem{JiangSu99}
X.~Jiang, H.~Su: On a simple unital projectionless \cstar-algebra.
\newblock Amer. J. Math. 121 (1999), no.~2, pp. 359--413.

\bibitem{KadisonRingrose67}
R.~V. Kadison, J.~R. Ringrose: Derivations and automorphisms of operator
  algebras.
\newblock Comm. Math. Phys. 4 (1967), pp. 32--63.

\bibitem{Kirchberg04}
E.~Kirchberg: Central sequences in \cstar-algebras and strongly purely infinite
  algebras.
\newblock Operator Algebras: The Abel Symposium 1 (2004), pp. 175--231.

\bibitem{KirchbergPhillips00}
E.~Kirchberg, N.~C. Phillips: Embedding of exact \cstar-algebras in the {C}untz
  algebra {$\CO_2$}.
\newblock J. reine angew. Math. 525 (2000), pp. 17--53.

\bibitem{KirchbergRordam14}
E.~Kirchberg, M.~R{\o}rdam: Central sequence \cstar-algebras and tensorial
  absoption of the {J}iang--{S}u algeba.
\newblock J. reine angew. Math. 695 (2014), pp. 175--214.

\bibitem{Liao16}
H.-C. Liao: A {R}okhlin type theorem for simple \cstar-algebras of finite
  nuclear dimension.
\newblock J. Funct. Anal. 270 (2016), pp. 3675--3708.

\bibitem{Liao17}
H.-C. Liao: Rokhlin dimension of {$\IZ^m$}-actions on simple \cstar-algebras.
\newblock Internat. J. Math. 28 (2017), no.~7.

\bibitem{MatuiSato12}
H.~Matui, Y.~Sato: {$\CZ$}-stability of crossed products by strongly outer
  actions.
\newblock Comm. Math. Phys. 314 (2012), no.~1, pp. 193--228.

\bibitem{MatuiSato12acta}
H.~Matui, Y.~Sato: Strict comparison and {$\CZ$}-absorption of nuclear
  \cstar-algebras.
\newblock Acta Math. 209 (2012), pp. 179--196.

\bibitem{MatuiSato14}
H.~Matui, Y.~Sato: {$\CZ$}-stability of crossed products by strongly outer
  actions {II}.
\newblock Amer. J. Math. 136 (2014), pp. 1441--1497.

\bibitem{MatuiSato14UHF}
H.~Matui, Y.~Sato: Decomposition rank of {UHF}-absorbing \cstar-algebras.
\newblock Duke Math. J. 163 (2014), pp. 2687--2708.

\bibitem{Nawata21}
N.~Nawata: {R}ohlin actions of finite groups on the {R}azak--{J}acelon algebra.
\newblock IMRN  (2021), no.~4, pp. 2991--3020.

\bibitem{Nawata23}
N.~Nawata: Equivariant {K}irchberg--{P}hillips type absorption for the
  {R}azak--{J}acelon algebra.
\newblock J. Funct. Anal. 285 (2023), no.~8.
\newblock Article 110088.

\bibitem{Nawata23preprint}
N.~Nawata: Strongly outer actions of certain torsion-free amenable groups on
  the {R}azak-{J}acelon algebra  (2023).
\newblock \urlprefix\url{https://arxiv.org/abs/2309.00934}.

\bibitem{OzawaSuzuki21}
N.~Ozawa, Y.~Suzuki: On characterizations of amenable \cstar-dynamical systems
  and new examples.
\newblock Selecta Math. 27 (2021).
\newblock Article no. 92.

\bibitem{Pedersen}
G.~K. Pedersen: \cstar-algebras and their automorphism groups.
\newblock Academic Press Inc. (2018).
\newblock Second Edition.

\bibitem{Sato10}
Y.~Sato: The {R}ohlin property for automorphisms of the {J}iang--{S}u algebra.
\newblock J. Funct. Anal. 259 (2010), no.~2, pp. 453--476.

\bibitem{Sato19}
Y.~Sato: Actions of amenable groups and crossed products of {$\CZ$}-absorbing
  \cstar-algebras.
\newblock Adv. Stud. Pure Math. 80 (2019), pp. 189--210.

\bibitem{SatoWhiteWinter15}
Y.~Sato, S.~White, W.~Winter: Nuclear dimension and {$\CZ$}-stability.
\newblock Invent. Math. 202 (2015), pp. 893--921.

\bibitem{Suzuki19}
Y.~Suzuki: Simple equivariant \cstar-algebras whose full and reduced crossed
  products coincide.
\newblock J. Noncommut. Geom. 13 (2019), pp. 1577--1585.

\bibitem{Suzuki21}
Y.~Suzuki: Equivariant {$\mathcal{O}_2$}-absorption theorem for exact groups.
\newblock Compos. Math. 157 (2021), no.~7, pp. 1492--1506.

\bibitem{Suzuki23-2}
Y.~Suzuki: Amenable actions on finite simple \cstar-algebras arising from flows
  on {P}imsner algebras  (2023).
\newblock \urlprefix\url{https://arxiv.org/abs/2305.13056}.

\bibitem{Suzuki23}
Y.~Suzuki: Every countable group admits amenable actions on stably finite
  simple \cstar-algebras.
\newblock Amer. J. Math., to appear  (2023).
\newblock \urlprefix\url{https://arxiv.org/abs/2204.04480}.

\bibitem{Szabo18ssa}
G.~Szab\'{o}: Strongly self-absorbing {$\mathrm{C}^*$}-dynamical systems.
\newblock Trans. Amer. Math. Soc. 370 (2018), pp. 99--130.

\bibitem{Szabo19ssa4}
G.~Szab{\'o}: Actions of certain torsion-free elementary amenable groups on
  strongly self-absorbing \cstar-algebras.
\newblock Comm. Math. Phys. 371 (2019), no.~1, pp. 267--284.

\bibitem{Szabo21si}
G.~Szab{\'o}: Equivariant property {(SI)} revisited.
\newblock Anal. PDE 14 (2021), no.~4, pp. 1199--1232.

\bibitem{SzaboWouters23-1}
G.~Szab{\'o}, L.~Wouters: Dynamical {McDuff}-type properties for group actions
  on von {N}eumann algebras  (2023).
\newblock \urlprefix\url{https://arxiv.org/abs/2301.11748}.

\bibitem{SzaboWouters23-2}
G.~Szab{\'o}, L.~Wouters: Equivariant property {G}amma and the tracial
  local-to-global principle for \cstar-dynamics  (2023).
\newblock \urlprefix\url{https://arxiv.org/abs/2301.12846}.

\bibitem{TomsWinter09}
A.~Toms, W.~Winter: The {E}lliott conjecture for {V}illadsen algebras of the
  first type.
\newblock J. Funct. Anal. 256 (2009), pp. 1311--1340.

\bibitem{TomsWhiteWinter15}
A.~S. Toms, S.~White, W.~Winter: {$\mathcal{Z}$}-stability and finite
  dimensional tracial boundaries.
\newblock Int. Math. Res. Not.  (2015), no.~10, pp. 2702--2727.

\bibitem{Winter16}
W.~Winter: Classifying crossed product \cstar-algebras.
\newblock Amer. J. Math. 138 (2016), pp. 793--820.

\bibitem{WinterZacharias10}
W.~Winter, J.~Zacharias: The nuclear dimension of \cstar-algebras.
\newblock Adv. Math. 224 (2010), no.~2, pp. 461--498.

\bibitem{Wouters23}
L.~Wouters: Equivariant {$\CZ$}-stability for single automorphisms on simple
  \cstar-algebras with tractable trace simplices.
\newblock Math. Z. 304 (2023).
\newblock Article no. 22.

\end{thebibliography}
\end{document}